\documentclass[12pt]{article}
\usepackage{latexsym}
\usepackage{amsmath}
\usepackage{amssymb}
\usepackage{amscd}
\usepackage{array}
\usepackage{comment}
\usepackage[all]{xy}
\usepackage{amsthm}
\usepackage{amsfonts}
\usepackage{enumerate}
\usepackage{hyperref}
\usepackage{authblk}
\usepackage{stackrel}
\newtheorem{theorem}{Theorem}[]
\newtheorem{proposition}[theorem]{Proposition}

\newtheorem{corollary}[theorem]{Corollary}

\theoremstyle{definition}

\newtheorem{remark}[theorem]{Remark}

\newcommand{\HH}{\mathcal H}

\newcommand{\Z}{\mathbf Z}

\newcommand{\Hol}{\mathrm{Hol}}

\newcommand{\Sym}{\operatorname{Sym}}
\newcommand{\Ker}{\operatorname{Ker}}

\newcommand{\End}{\operatorname{End}}

\newcommand{\Aut}{\operatorname{Aut}}

\newcommand{\Id}{\operatorname{Id}}

\newcommand{\Dic}{\operatorname{Dic}}

\title{Inducing braces and Hopf Galois structures}
\author{Teresa Crespo$^{(1)}$, Daniel Gil-Mu\~noz$^{(2)}$, Anna Rio$^{(3)}$ and Montserrat Vela$^{(3)}$}

\begin{document}

\maketitle

\footnotesize

\noindent
(1) Departament de Matem\`atiques i Inform\`atica, Universitat de Barcelona, Gran Via de les Corts Catalanes 585, 08007, Barcelona (Spain)

\noindent
(2) Charles University, Faculty of Mathematics and Physics, Department of Algebra, Sokolovska 83, 18600 Praha 8, Czech Republic

\noindent
(3) Departament de Matem\`atiques, Universitat Polit\`ecnica de Catalunya, Edifici Omega, Jordi Girona, 1-3, 08034, Barcelona (Spain)

\normalsize	
\begin{abstract}
Let $p$ be a prime number and let $n$ be an integer not divisible by $p$ and such that every group of order $np$ has a normal subgroup of order $p$. (This holds in particular for $p>n$.) We prove that left braces of size $np$ may be obtained as a semidirect product of the unique left brace of size $p$  and a left brace of size $n$. We give a method to determine all braces of size $np$ from the braces of size $n$ and certain classes of morphisms from the multiplicative group of these braces of size $n$ to $\Z_p^*$. From it we derive a formula giving the number of Hopf Galois structures of abelian type $\Z_p \times E$ on a Galois extension of degree $np$ in terms of the number of Hopf Galois structures of abelian type $E$ on a Galois extension of degree $n$. For a prime number $p\geq 7$, we apply the obtained results to describe all left braces of size $12p$ and determine the number of Hopf Galois structures of abelian type on a Galois extension of degree $12p$.
\end{abstract}

\noindent
{\small{\bf Keywords}{:}
Left braces, Holomorphs, regular subgroups, Hopf Galois structures.

\noindent
2020MSC: Primary: 16T05, 16T25, 12F10; Secondary: 20B35, 20B05, 20D20, 20D45.}

\let\thefootnote\relax\footnotetext{The first author was supported
by grant PID2019-107297GB-I00 (Ministerio de Ciencia, Innovaci\'on y Universidades).
The second author was supported by Czech Science Foundation, grant 21-00420M, and by Charles University Research Centre program UNCE/SCI/022.

Email addresses: teresa.crespo@ub.edu, daniel.gil-munoz@mff.cuni.cz, ana.rio@upc.edu, montse.vela@upc.edu }

\section{Introduction}

In \cite{R} Rump introduced an algebraic structure called brace to study
set-theoretic solutions of the Yang-Baxter equation. A left brace is a set $B$ with two operations $+$ and $\cdot$ such
that $(B, +)$ is an abelian group, $(B, \cdot)$ is a group and the brace relation is satisfied, namely,
$$a(b+c) = ab-a+ac,$$
for all $a, b, c \in B$. We call $N=(B, +$) the additive group and $G=(B, \cdot)$ the multiplicative group of the left brace. If $(B,+)$ is not abelian, the corresponding brace is called skew brace.

Let $B_1$ and $B_2$ be left braces. A map $f : B_1 \to B_2$ is said to be a brace morphism if $f(b+b') = f(b)+f(b')$ and $f(bb') = f(b)f(b')$ for all $b,b' \in B_1$. If $f$ is bijective, we say that $f$ is an isomorphism. In that case we say that the braces $B_1$ and $B_2$ are isomorphic.

In \cite{B}
Bachiller proved that given an abelian group $N$, there is
a bijective correspondence between left braces with additive group $N$, and regular subgroups of $\Hol(N)$ such that isomorphic left braces correspond to regular subgroups of $\Hol(N)$ which are conjugate by elements of $\Aut(N)$. In this way he established the connection between braces and Hopf-Galois structures on Galois field extensions.

In \cite{BNY} Lemma 2.1, it is proved that $\Aut(N)$, as a subgroup of $\Hol(N)$, is action-closed with respect to the conjugation action of $\Hol(N)$ on
the set of regular subgroups of $\Hol(N)$. Therefore, given an abelian group
$N$, the set of isomorphism classes of left braces with additive group $N$ is in bijective correspondence with the set of conjugacy classes of regular subgroups in $\Hol(N)$.

In \cite{AB} skew left braces of size $pq$ are classified, where $p>q$ are prime numbers. In \cite{D} a classification of left braces of order $p^2q$, where $p,q$  are prime numbers such that $q>p+1$ is given. In \cite{BNY} the following conjecture on the number $b(12p)$ of isomorphism classes of  left braces of size $12p$ is formulated.
	
	\begin{equation}\label{conj}
b(12p)= \left\{ \begin{array}{lll} 24 & \text{\ if } p\equiv 11 & \pmod{12}, \\
	28 & \text{\ if } p\equiv 5 & \pmod{12}, \\
	34 & \text{\ if } p\equiv 7 & \pmod{12}, \\
	40 & \text{\ if } p\equiv 1 & \pmod{12}.
	\end{array} \right.
	\end{equation}

We note that $b(24)=96, b(36)=46$ and $b(60)=28$ (see \cite{V}).

Let $B_1$ and $B_2$ be left braces. Then $B_1 \times B_2$ together with $+$ and $\cdot$ defined by
$$
(a,b)+(a',b')=(a+a',b+b')\quad (a,b)\cdot(a',b')=(aa',bb')
$$
is a left brace called the direct product of the braces $B_1$ and $B_2$.

Let $B_1$ and $B_2$ be left braces. Let $\tau:(B_2,\cdot)\to\Aut(B_1,+,\cdot)$
be a morphism of groups.
Consider in $B_1\times B_2$
the additive structure of the direct product $(B_1,+)\times (B_2,+)$
$$
(a,b)+(a',b')=(a+a',b+b')
$$
and the multiplicative structure of the semidirect product
$(B_1,\cdot)\rtimes_{\tau} (B_2,\cdot)$
$$
(a,b)\cdot(a',b')=(a\tau_b(a'),b b')
$$
Then, we get a left brace,  which is called the semidirect product of the left braces $B_1$ and $B_2$ via $\tau$.

A Hopf Galois structure on a finite extension of fields $K/k$ is a pair $(\HH,\mu)$ where $\HH$ is a finite cocommutative $k$-Hopf algebra  and $\mu$ is a
Hopf action of $\HH$ on $K$, i.e a $k$-linear map $\mu: \HH\to
\End_k(K)$ giving $K$ a left $\HH$-module algebra structure and inducing a bijection $K\otimes_k \HH\to\End_k(K)$.
Hopf Galois extensions were introduced by Chase and Sweedler in \cite{C-S}.
For a Galois field extension $K/k$ with Galois group $G$, Greither and
Pareigis \cite{G-P} give a bijective correspondence between Hopf Galois structures on $K/k$ and regular subgroups $N$ of $\Sym(G)$ normalized by $\lambda (G)$, where
$\lambda$ denotes left translation. For a given Hopf Galois structure on $K/k$, we will refer to the isomorphism class of the corresponding group $N$ as the type of the Hopf Galois
structure.
By Byott translation theorem \cite{By}, a correspondence is established between regular subgroups $N$ of $\Sym(G)$ normalized by $\lambda (G)$ and regular subgroups of the holomorph $\Hol(N)=N\rtimes \Aut N$. As a corollary, Byott obtains the following formula.

\begin{proposition}[\cite{By} Corollary to Proposition 1]\label{Bformula} Let $K/k$ be a Galois extension with Galois group $G$. Let $N$ be a group of order $|G|$. Let $a(N,G)$ denote the number of Hopf Galois structures of type $N$ on $K/k$ and let $b(N,G)$ denote the number of regular subgroups of $\Hol(N)$ isomorphic to $G$. Then

$$a(N,G)= \dfrac{|\Aut G|}{|\Aut N|} \, b(N,G).$$
\end{proposition}

In \cite{RCGV} we have proved that every brace of size $8p$, for a prime number $p \neq 3,7$, is the semidirect product of a brace of size 8 and the unique brace of size $p$. We have used this fact to determine all braces of size $8p$. In this paper we generalize this result to braces of size $np$, where $p$ is a prime number and $n$ an integer not divisible by $p$ and such that every group of order $np$ has a normal subgroup of order $p$. We note that these hypothesis hold in particular for $p>n$. From our result on braces we derive a formula giving the number of Hopf Galois structures of abelian type $\Z_p \times E$ on a Galois extension of degree $np$. For a prime number $p\geq 7$, we apply the obtained results to describe all left braces of size $12p$ and determine the number of Hopf Galois structures of abelian type on a Galois extension of degree $12p$.
As a consequence of our classification of left braces of size $12p$, for $p$ a prime number, $p \geq 7$, we establish the validity of conjecture (\ref{conj}).

From now on, $p$ and $n$ will always satisfy the following hypothesis.

\vspace{0.4cm}

(H):  {\it $p$ is a prime number and $n$ an integer such that $p$ does not divide $n$

\qquad and each group of order $np$ has a normal subgroup of order $p$.}

\vspace{0.4cm}

\section{Braces of size $np$}

The following proposition is a generalization of \cite{RCGV}, Proposition 1.

\begin{proposition}\label{braceprod} Let $p$ be a prime and $n$ an integer such that $p$ does not divide $n$ and each group of order $np$ has a normal subgroup of order $p$. Then every left brace of size $np$ is a direct or semidirect product of the trivial brace of size $p$ and a left brace  of size $n$.
\end{proposition}

\begin{proof}
Let $B$ be a left brace of size $np$ with additive group $N$ and multiplicative group $G$. Then, by the Schur-Zassenhaus theorem, $N=\Z_p\times E$ with $E$ an abelian group of order $n$ and $G=\Z_p\rtimes_{\tau} F$ with $F$ a group of order $n$ and $\tau:F\to\Aut(\Z_p)$ a group morphism (the trivial one giving the direct product). Let us observe that, since we are working with the trivial brace of size $p$, the group of brace automorphisms is the classical group $\Aut(\Z_p)\simeq Z_p^*$.

Then, for $(a_1,a_2),(b_1,b_2), (c_1,c_2) \in B$,
$$\begin{matrix}
(a_1,a_2)((b_1,b_2)+(c_1,c_2))+(a_1,a_2)=
(a_1,a_2)(b_1+c_1,b_2+c_2)+(a_1,a_2)=\\
=(a_1+\tau_{a_2}(b_1+c_1)+a_1, a_2(b_2+c_2)+a_2).
\end{matrix}
$$
On the other hand,
$$\begin{matrix}
(a_1,a_2)(b_1,b_2)+(a_1,a_2)(c_1,c_2)
=(a_1+\tau_{a_2}(b_1)+a_1+\tau_{a_2}(c_1), a_2b_2+a_2c_2).
\end{matrix}
$$
Therefore, from the brace condition of $B$ we obtain an equality in the second component which tells us that we have a brace $B'$ of size $n$ with  additive group $E$ and multiplicative group $F$. Then, $B$ is the semidirect product via $\tau$ of the trivial brace with group $\Z_p$ and this brace $B'$.
\end{proof}

\begin{remark} The third Sylow theorem gives that the hypothesis in Proposition \ref{braceprod} are satisfied in particular when $p>n$.
\end{remark}

As a corollary to Proposition \ref{braceprod}, we obtain that for each brace of size $n$, there is a left brace of size $np$ which is the direct product of the unique brace of size $p$ and the given brace of size $n$. The braces of size $np$ which are a semidirect product of the unique brace of size $p$ and a brace of size $n$ are determined by the following proposition, which generalizes Proposition 4 in \cite{RCGV}.

\begin{proposition}\label{determsemid}
  Let $p$ be a prime and $n$ an integer such that $p$ does not divide $n$ and each group of order $np$ has a normal subgroup of order $p$. Let $N=\Z_p\times E$ be an abelian group of order $np$.

 The conjugacy classes of regular subgroups of $\Hol(N)$ are in one-to-one correspondence with couples $(F,\tau)$ where $F$ runs over a set of representatives of conjugacy classes of regular subgroups of $\Hol(E)$ and $\tau$ runs over representatives of classes of group morphisms $\tau:F\to \Aut(\Z_p)$ under the relation
 $
 \tau\simeq \tau'$ if and only if $\tau=\tau'\circ \Phi_{\nu}|_F$
 where $\nu\in \Aut(E)$ and $\Phi_\nu$ is the corresponding inner automorphism of $\Hol(E)$.
\end{proposition}

\begin{proof} As in Proposition \ref{braceprod}, we may apply the Schur-Zassenhaus theorem and obtain that groups of order $np$ are semidirect products $G=\Z_p\rtimes_{\tau} F$ with $F$ a group of order $n$ and $\tau:F\to\Aut(\Z_p)$ a group morphism.

For a given couple $(F,\tau)$ the semidirect product is
$$
G=\Z_p\rtimes_{\tau} F=
\{((m,\tau(f)), f)\mid\  m\in\Z_p,\  f\in F\}\subseteq
(\Z_p\rtimes \Z_p^*)\times \Hol(E)=\Hol(N)
$$
and
the action on $N$ is given by
$((m,k),f)(z,x)=(m+kz, fx)$. Since $G$ contains $\Z_p$, we have transitivity in the first component and $G$ is regular in $\Hol(N)$ if and only if $F$ is regular in $\Hol(E)$.

Let us describe inner automorphisms of $\Hol(N)=(\Z_p\rtimes \Z_p^*)\times (E\rtimes \Aut(E))$. We write elements in
$\Hol(N)$ as $(m,k,a,\sigma)$ accordingly. Since we are dealing with regular subgroups, we just have to consider conjugation by elements
$(i,\nu)\in\Aut(N)=\Z_p^*\times \Aut(E)$. Let $\Phi_{(i,\nu)}$ be the inner automorphism corresponding to $(i,\nu)$ inside $\mathrm{Hol}(N)$. Then,
$$
\begin{array}{lll} \Phi_{(i,\nu)}(m,k,a,\sigma)&=&
(0,i,0,\nu)(m,k,a,\sigma)(0,i,0,\nu)^{-1} \\ [10pt]
&=&
(im,ik,\nu(a),\nu\sigma)(0,i^{-1},0,\nu^{-1})\\ [10pt]
&=&
(im,k,\nu(a),\nu\sigma\nu^{-1})
\end{array}
$$
If we work in $\Hol(E)$, conjugation by $\nu\in\Aut(E)$ is
$$
\Phi_{\nu}(a,\sigma)=(0,\nu)(a,\sigma)(0,\nu^{-1})=(\nu(a),\nu\sigma\nu^{-1}).
$$

Let
$
G=\Z_p\rtimes_{\tau} F=
\{(m,\tau(a,\sigma), a,\sigma)\mid m\in\mathbb{Z}_p,\,(a,\sigma)\in F\}.
$
Then,
$$
\Phi_{(i,\nu)}(G)
=\{ (im, \tau(a,\sigma),\nu(a),\nu\sigma\nu^{-1})\mid m\in\mathbb{Z}_p,\,(a,\sigma)\in F\}.
$$
Since $i\in \Z_p^*$, $im$ runs over $\Z_p$ as $m$ does.
Therefore, if $(F',\tau')$ is another pair, we have
$$\Phi_{(i,\nu)}(G)=\Z_p\rtimes_{\tau'} F'\iff
F'=\Phi_{\nu}(F),\mbox{ and }
\tau=\tau'\circ \Phi_{\nu}|_F.
$$
Let us observe that in that case $\ker\tau'=\Phi_{\nu}(\ker\tau)$.
\end{proof}

\section{Hopf Galois structures on a Galois field extension of degree $np$}

From Proposition \ref{determsemid} we obtain the following corollary.

\begin{corollary}\label{cor}
Let $E$  be  a group of order $n$, $N=\Z_p \times E$. Let $F$ be a regular subgroup of $\Hol(E)$ and $\tau:F \to \Z_p^*$ a group morphism. The length of the conjugacy class of the regular subgroup of $\Hol(N)$ corresponding to $(F,\tau)$ is equal to the length of the conjugacy class of $F$ in $\Hol(E)$ times the number of morphisms from $F$ to $\Z_p^*$ equivalent to $\tau$ under the relation defined in Proposition \ref{determsemid}.
\end{corollary}

Using this corollary, we shall determine, the number of regular subgroups of the holomorph of $N$. Applying then Byott's formula (Proposition \ref{Bformula}), we shall obtain the number of Hopf Galois structures of abelian type on a Galois extension of degree $np$. We note that all these Galois structures are induced, in the sense of \cite{CRV}, by Theorem 9 in loc. cit.
In order to apply Byott's formula, we determine the automorphisms of a semidirect product $\Z_p \rtimes_{\tau} F$.

Let $G=\Z_p \rtimes F$, with $F$ a group of order $n$. By the Schur-Zassenhaus theorem, any subgroup of $G$ of order equal to  $|F|$ is conjugate to $F$. We assume that the semidirect product is not direct, then $F$ has exactly $p$ conjugates, namely $F_i:=(i,1_F)F(-i,1_F), 0\leq i \leq p-1$.  If $\varphi$ is an automorphism of $G$, then $\varphi(\Z_p)=\Z_p$ and $\varphi(F)$ is a subgroup of $G$ isomorphic to $F$. We have then $\varphi(F)=F_i$ for some $i$. Let

$$S= \{ \varphi \in \Aut G \, : \, \varphi(F)=F \}.$$

\noindent
Clearly $S$ is a subgroup of $\Aut(G)$. Let $C_i$ denote conjugation by $(i,1)$ in $\Aut(G)$. Then $\{ C_i \}_{0\leq i \leq p-1}$ is a transversal of $S$ in $\Aut (G)$, hence $|\Aut (G)|=p|S|$.

We give now a characterization of $S$ in terms of $\Aut \Z_p, \Aut F$ and the morphism $\tau:F \rightarrow \Aut \Z_p\simeq \Z_p^*$ defining the semidirect product $\Z_p \rtimes F$.

\begin{proposition}\label{prop} The image of the injective map

$$S \rightarrow \Aut \Z_p \times \Aut F, \quad \varphi \mapsto (\varphi_{|\Z_p},\varphi_{|F})$$

\noindent
is precisely the set of pairs $(f,g) \in \Aut \Z_p \times \Aut F$ such that $\tau g=\tau$.

\end{proposition}

\begin{proof} Let $\varphi \in \Aut G$. For $x \in F$, $1 \in \Z_p$, we have $x1=\tau(x)x$. Applying $\varphi$ to this equality, we get $\varphi(x) \varphi(1)=\varphi(\tau(x)) \varphi(x)$. Now, since $\varphi(x) \in F$ and $\varphi(1) \in \Z_p$, we have $\varphi(x) \varphi(1)=\tau (\varphi(x))\varphi(1) \varphi(x)$. We obtain then the equality $\varphi(\tau(x)) = \tau (\varphi(x))\varphi(1)$. This implies $\varphi_{|\Z_p} \tau(x)=\tau (\varphi_{|F}(x)) \varphi_{|\Z_p}$ in $\Aut \Z_p$. Since $\Aut Z_p$ is commutative, we obtain $\tau = \tau  \varphi_{|F}$.

Reciprocally, let $(f,g) \in \Aut \Z_p \times \Aut F$ such that $\tau g=\tau$. We define a map $\varphi$ from $\Z_p \times  F$ to $\Z_p \times F$ by $\varphi(i,x)=(f(i),g(x))$. Now $\varphi$ is an automorphism of $\Z_p \rtimes_{\tau} F$ if and only if $\varphi((i,x)(j,y))=\varphi((i,x))\varphi((j,y))$, equivalently $(f(i+\tau(x)j),g(xy)))=(f(i),g(x))(f(j),g(y))=(f(i)+\tau(g(x))f(j),g(x)g(y))$. Since $g$ is an automorphism, the two second components coincide. Since $f$ is an automorphism, the equality of the first components is equivalent to $f( \tau(x)j)=\tau(g(x))f(j)$ for all $j$, equivalently $f\tau(x)=\tau(g(x))f$, for all $x \in F$, which is fulfilled, since $\tau g=\tau$ and $\Aut \Z_p$ is commutative.
\end{proof}

\begin{corollary}\label{autf}
For $G=\Z_p \rtimes_{\tau}F$, we have $|\Aut G|=p(p-1)|S_0|$, where $S_0=\{ g \in \Aut F \, | \, \tau g=\tau \}$.
\end{corollary}

\begin{proof}
From the proposition we obtain clearly $S=\Aut \Z_p \times S_0$,  hence $|\Aut G|=p|S|=p(p-1)|S_0|.$
\end{proof}

\section{Braces of size $12p$: direct products}

There are five groups of order 12, up to isomorphism, two abelian ones $C_{12}$ and $C_6\times C_2$ and three non-abelian ones, the alternating group $A_4$, the dihedral group $D_{2\cdot 6}$ and the dicyclic group $\Dic_{12}$. By computation with Magma, we obtain that the number of left braces with additive group $E$ and multiplicative group $F$ is as shown in the following table.

\vspace{0.5cm}
\begin{center}
\begin{tabular}{|c||c|c|c|c|c|}
\hline
$E \backslash F$ & \, \, $C_{12}$ \, \, & $C_6\times C_2$ & \, \, $A_4$ \, \, \, & \, \, $D_{2\cdot 6}$ \, \, & \,  $\Dic_{12}$ \, \, \\
\hline
\hline
$C_{12}$ &1&1&0&2& 1\\
\hline
$C_6\times C_2$ &1&1&1&1& 1\\
\hline
\end{tabular}
\end{center}

\vspace{0.5cm}
For $p$ a prime number, $p \geq 7$, the Sylow theorems give that a group $G$ of order $12p$ has a normal subgroup $H_p$ of order $p$. We obtain then the following corollary to Proposition \ref{braceprod}.

\begin{corollary}\label{corbraceprod} Let $p\geq 7$ be a prime.
Every left brace of size $12p$ is a direct or semidirect product of the trivial brace of size $p$ and a left brace  of size $12$.
\end{corollary}

From the description of the braces of size 12 and the definition of direct product of braces we obtain the following result.

\begin{proposition}\label{direct} For a prime number $p$, there are 10 left braces of size $12p$ which are direct product of the unique brace of size $p$ and a brace of size $12$.
\end{proposition}

\section{Braces of size $12p$: semidirect products}\label{braces}

For $p\geq 7$ and $n=12$, the hypothesis of Proposition \ref{determsemid} are satisfied and we shall apply it to determine the braces of size $12p$ which are semidirect products of the unique brace of size $p$ and a brace of size 12. To this end, we shall consider the braces of order 12 with additive group $E$ and multiplicative group $F$ and determine the classes of the morphisms $\tau:F\to \Aut(\Z_p)$ under the relation described in Proposition \ref{determsemid}. We note that finding all such morphisms $\tau$ reduces to consider the normal subgroups $F'$ of $F$ such that $F/F'$ is a cyclic group $C$ whose order divides $p-1$ and taking into account the automorphisms of $C$. From now on, the kernel of $\tau$ will be referred to as the kernel of the brace (or conjugation class of regular subgroups) determined by the pair $(F,\tau)$.

\begin{remark}[Description of the holomorphs]\label{hol}
We consider now the abelian groups of order 12, that is, $E=C_{12}$ and $E=C_6\times C_2$ and describe $\Hol(E)$ in each case.

For $E=C_{12}=\Z_{12}$, we have $\Aut(\Z_{12})=\Z_{12}^*=\{1,5,7,11\}\simeq C_2\times C_2$ and $\Hol(\Z_{12})=\{ (x,l) \, : \, x \in \Z_{12}, l \in \Z_{12}^* \}$ with product given by $(x,l)(y,m)=(x+ly,lm)$.

For $E=C_6\times C_2$, we have $\Aut(E) \simeq D_{2\cdot 6}$. We write $C_6\times C_2=\langle a \rangle \times \langle b \rangle$ and consider the automorphisms $\rho,\sigma$ of $E$ defined by

$$\begin{array}{llll} \rho:& a& \mapsto & a^5 b \\ &b&\mapsto &a^3 \end{array} \qquad \begin{array}{llll} \sigma:& a& \mapsto & a^5 \\ &b&\mapsto &a^3b \end{array}.$$

\noindent
We may check that $\rho$ has order 6, $\sigma$ has order 2 and $\sigma \rho \sigma=\rho^{-1}$, hence $\Aut(E)=\langle \rho,\sigma \rangle$.
We have $\Hol(E)=\{ (x,\varphi) \, : \, x \in E, \varphi \in \Aut E \}$ with product defined by $(x,\varphi)(y,\psi)=(x\varphi(y),\varphi\psi)$.

We shall use the descriptions above along this section.
\end{remark}

\subsection{$F=C_{12}$}

Let us write $F=\langle x \rangle$. We determine now the possible morphisms $\tau:F \rightarrow \Z_p^*$. To be used in Section \ref{HG12p}, we compute $S_0(\tau)=\{ g \in \Aut F \, | \, \tau g=\tau \}$. We have $\Aut C_{12} \simeq \Z_{12}^* =\{ 1,5,7,11 \}$.

\begin{enumerate}[1)]
\item There is a unique morphism $\tau:F \rightarrow \Z_p^*$ with kernel of order 6, namely the one sending the generator $x$ of $F$ to $-1$. We have $S_0(\tau)=\Aut F$.
\item When $p\equiv 1 \pmod 4$, $\Z_p^*$ has a (unique) subgroup of order 4. Let $\zeta_4$ be a generator of it. We may define two morphisms from $F$ to $\Z_p^*$ with a kernel of order 3, namely

    $$\tau_1: x \mapsto \zeta_4, \quad \tau_2: x \mapsto \zeta_4^{-1}.$$

\noindent
We have $S_0(\tau_1)=S_0(\tau_2)=\{ 1,5 \}$.

\item When $p\equiv 1 \pmod 6$, $\Z_p^*$ has a (unique) subgroup of order 6. Let $\zeta_6$ be a generator of it. We may define two morphisms from $F$ to $\Z_p^*$ with a kernel of order 2, namely

    $$\tau_1: x \mapsto \zeta_6, \quad \tau_2: x \mapsto \zeta_6^{-1}$$

and two morphisms from $F$ to $\Z_p^*$ with a kernel of order 4, namely

    $$\tau_3: x \mapsto \zeta_6^2, \quad \tau_4: x \mapsto \zeta_6^{-2}.$$

\noindent
We have $S_0(\tau_1)=S_0(\tau_2)=S_0(\tau_3)=S_0(\tau_4)=\{1,7\}$.

\item When $p\equiv 1 \pmod {12}$, $\Z_p^*$ has a (unique) subgroup of order 12. Let $\zeta_{12}$ be a generator of it. We may define four morphisms from $F$ to $\Z_p^*$ with a trivial kernel, namely

    $$\begin{array}{lll} \tau_1: x \mapsto \zeta_{12}, && \tau_2: x \mapsto \zeta_{12}^{5},\\[10pt]
       \tau_3: x \mapsto \zeta_{12}^{-5}, && \tau_4: x \mapsto \zeta_{12}^{-1}.\end{array}$$

\noindent
We have $S_0(\tau_1)=S_0(\tau_2)=S_0(\tau_3)=S_0(\tau_4)=\{ 1 \}$.

\end{enumerate}

\subsubsection*{Case $E=C_{12}$}

If $E=C_{12}$, we may take $F=\langle (1,1) \rangle \subset \Hol(E)$, i.e. we have now $x=(1,1)$. We determine the conjugation relations between the  morphisms $\tau:F \rightarrow \Z_p^*$.

\begin{enumerate}[1)]
\item We consider the two morphisms from $F$ to $\Z_p^*$ with a kernel of order 3.
We observe that $\tau_2(-1,1)=\tau_2((1,1)^{-1})=\zeta_4$, hence $\tau_1=\tau_2 \Phi_{-1}$ and we obtain then one brace.
\item We consider the two morphisms from $F$ to $\Z_p^*$ with a kernel of order 2 and the two with a kernel of order 4.
We have $\tau_1=\tau_2 \Phi_{-1}$ and $\tau_3=\tau_4 \Phi_{-1}$ and obtain then two braces.
\item We consider the four morphisms from $F$ to $\Z_p^*$ with a trivial kernel.
We observe that $(1,1)^5=(5,1), (1,1)^{-5}=(-5,1), (1,1)^{-1}=(-1,1)$, hence $\tau_1=\tau_2 \Phi_{5}=\tau_3 \Phi_{-5}= \tau_4 \Phi_{-1}$ and obtain then one brace.
\end{enumerate}

We state the obtained result in the following proposition.

\begin{proposition}
 Let $p\geq 7$ be a prime number. We count the left braces with additive group
     $\Z_p\times \Z_{12}$ and multiplicative group $\Z_p\rtimes \Z_{12}$.
 \begin{enumerate}[1)]
     \item If $p\equiv 11 \pmod{12}$ there are 2 such braces.  One of them is a direct product and the second one has a kernel of order 6.
      \item If $p\equiv 5 \pmod{12}$ there are 3 such braces.  Two of them are as in 1) and the third one has a kernel of order 3.
       \item If $p\equiv 7 \pmod{12}$ there are 4 such braces. Two of them are as in 1) and the other two have kernels of orders 2 and 4, respectively.
            \item If $p\equiv 1 \pmod{12}$ there are 6 such braces. One of them is a direct product and the other five have kernels of orders 6,4,3,2,1, respectively.

 \end{enumerate}
\end{proposition}

\subsubsection*{Case $E=C_6\times C_2$}

 For $E=C_6\times C_2$, we use the notations in Remark \ref{hol}. We may take $F=\langle (ab,\varphi) \rangle \subset \Hol(E)$, where $\varphi$ is the order 2 automorphism defined by $\varphi(a)=a, \varphi(b)=a^3b$, i.e. $\varphi=\rho^3\sigma$. We may check that $F$ is indeed a cyclic group of order $12$ and a regular subgroup of $\Hol(E)$. We have now $x=(ab,\varphi)$. We determine the conjugation relations between the morphisms $\tau:F \rightarrow \Z_p^*$.

\begin{enumerate}[1)]
\item For the two morphisms from $F$ to $\Z_p^*$ with a kernel of order 3, we observe that $(ab,\varphi)^{-1}=(\varphi(a^{-1}b),\varphi)=(a^2b,\varphi)$, hence $\tau_2(a^2b,\varphi)=\zeta_4$. We have then  $\tau_1=\tau_2 \Phi_{\sigma}$, since $\Phi_{\sigma}(ab,\rho^3\sigma)=\sigma (ab,\rho^3\sigma) \sigma^{-1}=(\sigma(ab),\sigma(\rho^3\sigma)\sigma)=(a^2b,\rho^3 \sigma)$. We obtain then one brace.
\item For the two morphisms from $F$ to $\Z_p^*$ with a kernel of order 2 and the two with a kernel of order 4, as in the preceding case, we have $\tau_1=\tau_2 \Phi_{\sigma}$ and $\tau_3=\tau_4 \Phi_{\sigma}$ and obtain then two braces.
\item For the four morphisms from $F$ to $\Z_p^*$ with a trivial kernel, we observe that $(ab,\varphi)^5=(a^5b,\varphi), (ab,\varphi)^{-5}=(a^4b,\varphi), (ab,\varphi)^{-1}=(a^2b,\varphi)$, hence $\tau_1=\tau_2 \Phi_{\rho^3}=\tau_3 \Phi_{\rho^3\sigma}= \tau_4 \Phi_{\sigma}$ and we obtain then one brace.
\end{enumerate}

We state the obtained result in the following proposition.

\begin{proposition}
 Let $p\geq 7$ be a prime number. We count the left braces with additive group
     $\Z_p\times \Z_{6}\times \Z_2$ and multiplicative group $\Z_p\rtimes \Z_{12}$.
 \begin{enumerate}[1)]
     \item If $p\equiv 11 \pmod{12}$ there are 2 such braces. One of them is a direct product and the second one has a kernel of order 6.
      \item If $p\equiv 5 \pmod{12}$ there are 3 such braces. Two of them are as in 1) and the third one has a kernel of order 3.
       \item If $p\equiv 7 \pmod{12}$ there are 4 such braces. Two of them are as in 1) and the other two have kernels of orders 2 and 4, respectively.
            \item If $p\equiv 1 \pmod{12}$ there are 6 such braces. One of them is a direct product and the other five have kernels of orders 6,4,3,2,1, respectively.

 \end{enumerate}
\end{proposition}

\subsection{$F=C_6\times C_2$}

Let us write $F=\langle x, y \rangle$, with $x$ of order 6, $y$ of order 2. We determine now the possible morphisms $\tau:F \rightarrow \Z_p^*$. To be used in Section \ref{HG12p}, we compute $S_0(\tau)=\{ g \in \Aut F \, | \, \tau g=\tau \}$. We use the determination of $\Aut F$ given in Remark \ref{hol}.

\begin{enumerate}[1)]
\item There are three morphisms from $F$ to $\Z_p^*$ with kernel of order 6, namely

$$\begin{array}{lllr} \tau_1:&x & \mapsto & 1 \\ & y & \mapsto & -1 \end{array}, \quad \begin{array}{lllr} \tau_2:&x & \mapsto & -1 \\ & y & \mapsto & -1 \end{array}, \quad \begin{array}{lllr} \tau_3:&x & \mapsto & -1 \\ & y & \mapsto & 1 \end{array},$$

\noindent
with kernels $\langle x \rangle, \langle xy \rangle, \langle x^2y \rangle$, respectively. We have $S_0(\tau_1)= \langle \rho^3,\sigma \rangle, S_0(\tau_2)= \langle \rho^3,\rho^2\sigma \rangle, S_0(\tau_3)= \langle \rho^3,\rho\sigma \rangle$.

\item In order to have a morphism $\tau$ with $\Ker \tau$ of order 2 or 4, it is necessary that $p\equiv 1 \pmod 6$. In this case, let $\zeta_6$ be a generator of the unique subgroup of order 6 of $\Z_p^*$ . We may define six morphisms from $F$ to $\Z_p^*$ with a kernel of order 2, namely

$$\begin{array}{rlrll }
   \tau_1: &  x \mapsto \zeta_6  \qquad \qquad \qquad & \tau_2: & x \mapsto \zeta_6^{-1} \qquad \qquad &\textrm{ with } \Ker \tau=<y> \\[2mm]
           &y \mapsto 1 &    & y \mapsto 1,
    \end{array}$$
$$\begin{array}{rlrll}
   \tau_3: &  x \mapsto \zeta_6^2  \qquad \qquad \qquad & \tau_4: & x \mapsto \zeta_6^{-2} \qquad \qquad &\textrm{ with } \Ker \tau=<x^3>\\[2mm]
           &y \mapsto \zeta_6^3  &    & y \mapsto \zeta_6^3, &
    \end{array}$$
 $$\begin{array}{rlrll}
   \tau_5: &  x \mapsto \zeta_6  \qquad \qquad \qquad & \tau_6: & x \mapsto \zeta_6^{-1} \qquad \qquad & \textrm{ with } \Ker \tau=<x^3 y>\\[2mm]
           &y \mapsto \zeta_6^3 &    & y \mapsto \zeta_6^3. &
    \end{array}$$

\noindent
We have $S_0(\tau_1)=S_0(\tau_2)=\langle \rho \sigma \rangle, S_0(\tau_3)=S_0(\tau_4)=\langle \rho^3 \sigma \rangle, S_0(\tau_5)=S_0(\tau_6)=\langle \rho^5 \sigma \rangle$. We may further define two morphisms from $F$ to $\Z_p^*$ with a kernel of order 4, namely
       $$\begin{array}{rlrl}
   \tau_1: & x \mapsto \zeta_6^2  \qquad \qquad \qquad & \tau_2: & x \mapsto \zeta_6^{-2}\\[2mm]
           &y \mapsto 1 &    & y \mapsto 1.
    \end{array}$$

\noindent
We have $S_0(\tau_1)=S_0(\tau_2)=\langle \rho^2, \rho \sigma \rangle$.

\end{enumerate}

\subsubsection*{Case $E=C_{12}$}

We know that in $\Hol(C_{12})$ there is only a regular subgroup isomorph to $F$. We may take

$$F=\langle \alpha=(2,1),\beta=(3,7) \rangle \subset \Hol(E)$$

\noindent
following the notation in Remark \ref{hol}.

The element $\alpha$ has order 6, the element $\beta$ has order 2, they commute with each other and generate a regular subgroup of order 12. We have now $x=\alpha, y=\beta$.
We determine the conjugation relations between the morphisms $\tau:F \rightarrow \Z_p^*$.

\begin{enumerate}[1)]
\item For the morphisms from $F$ to $\Z_p^*$ with kernel of order 6, we have $\tau_2=\tau_3 \Phi_{-1}$ and $\tau_1$ is not conjugated to the other two, since the second component of $\alpha$ is different from those of $\alpha \beta$ and $\alpha^2 \beta$. We obtain then two braces.

\item For the morphisms from $F$ to $\Z_p^*$ with a kernel of order 4, we observe that $\tau_2\Phi_{11}(\alpha)=\zeta_3$ and $\tau_2\Phi_{11}(\beta)=1$, hence $\tau_1=\tau_2 \Phi_{11}$ and we obtain then a unique brace.

\item For the morphisms from $F$ to $\Z_p^*$ with a kernel of order 2, we observe that $\tau_2=\tau_1 \Phi_{5}$, $\tau_5=\tau_1 \Phi_{7}$, $\tau_6=\tau_1 \Phi_{-1}$ and  $\tau_4=\tau_3 \Phi_{-1}$. So we obtain only two braces (determined by $\tau_1$ and $\tau_3$).

\end{enumerate}

We state the obtained result in the following proposition.

\begin{proposition}
 Let $p\geq 7$ be a prime number. We count the left braces with additive group
     $\Z_p\times C_{12}$ and multiplicative group $\Z_p\rtimes (C_{6}\times C_2)$.
 \begin{enumerate}[1)]
     \item If $p\equiv 11 \pmod{12}$ there are 3 such braces. One of them is a direct product and the other two have a kernel of order 6.
     \item If $p\equiv 7 \pmod{12}$ there are 6 such braces. One of them is a direct product, two have kernel of order 6, two have kernels of order 2 and one has kernel of order 4.
   \item If $p\equiv 5 \pmod{12}$ there are 3 such braces. One of them is a direct product and the other two have a kernel of order 6.
  \item If $p\equiv 1 \pmod{12}$  there are 6 such braces. One of them is a direct product, two have kernel of order 6, two  have kernels of orders 2 and one  has kernel of order 4.

 \end{enumerate}
\end{proposition}

\subsubsection*{Case $E=C_6\times C_2$}

If $E=C_6\times C_2$, we may take $F=\langle (a,\Id),(b,\Id) \rangle \subset \Hol(E)$, following the notation of Remark \ref{hol}. We may check that $F$ is indeed a regular subgroup of order $12$  of $\Hol(E)$ isomorph to $C_6\times C_2$. We have now $x=(a,\Id), y=(b,\Id)$. We determine the conjugation relations between the morphisms $\tau:F \rightarrow \Z_p^*$.

\begin{enumerate}[1)]
\item For the morphisms from $F$ to $\Z_p^*$ with kernel of order 6, we have $\tau_1=\tau_2 \Phi_{\rho^4}=\tau_3 \Phi_{\rho^5}$. We obtain then one brace.

\item For the morphisms from $F$ to $\Z_p^*$ with a kernel of order 4, we observe that  $\tau_1=\tau_2 \Phi_{\rho^3}$ and obtain then a unique brace.

\item For the morphisms from $F$ to $\Z_p^*$ with a kernel of order 2,
we observe that $\tau_6=\tau_1 \Phi_{\rho}=\tau_2 \Phi_{\rho^4}=\tau_3 \Phi_{\rho^2}=\tau_4 \Phi_{\sigma \rho^2}=\tau_5 \Phi_{\rho^3}$. So we obtain only one brace.

\end{enumerate}

We state the obtained result in the following proposition.

\begin{proposition}
 Let $p\geq 7$ be a prime number. We count the left braces with additive group
     $\Z_p\times (C_{6}\times C_2)$ and multiplicative group $\Z_p\rtimes (C_{6}\times C_2)$.
 \begin{enumerate}[1)]
     \item If $p\equiv 11 \pmod{12}$ there are 2 such braces. One of them is a direct product and the second one has a kernel of order 6.
     \item If $p\equiv 7 \pmod{12}$ there are 4 such braces. One of them is a direct product, and the other three have kernels of orders 2, 3 and 4, respectively.
   \item If $p\equiv 5 \pmod{12}$ there are 2 such braces. One of them is a direct product and the second one has a kernel of order 6.
  \item If $p\equiv 1 \pmod{12}$  there are 4 such braces. One of them is a direct product, and the other three have kernels of orders 2, 3 and 4, respectively.

 \end{enumerate}
\end{proposition}

\subsection{$F=A_4$}

This case only occurs for $E=C_6\times C_2$. We use the notation of Remark \ref{hol} for the generators of $\Hol(E)$. We have $A_4=V_4\rtimes C_3$ and we may take $F=\langle a^3, b, (a^4,\rho^2) \rangle \subset \Hol(E)$, since $a^3, b$ are order 2 elements commuting between them and $(a^4,\rho^2)$ has order 3 and satisfies $(a^4,\rho^2)a^3(a^4,\rho^2)^{-1}=b, (a^4,\rho^2)b(a^4,\rho^2)^{-1}=a^3b$. We may further check that $F$ is a regular subgroup of $\Hol(E)$. Since $V_4$ is the unique proper nontrivial normal subgroup of $A_4$, we have that a nontrivial morphism from $F$ to $\Z_p^*$ has image a cyclic group of order $3$. We have then two cases.

\begin{enumerate}[1)]
\item If $p \not \equiv 1 \pmod{3}$, the unique morphism from $F$ to $\Z_p^*$ is the trivial one and there is just one brace with additive group $\Z_p\times \Z_{6}\times \Z_2$ and multiplicative group $\Z_p\rtimes A_4$, the one whose multiplicative group is a direct product.
\item If $p \equiv 1 \pmod{3}$, let $\zeta_3$ be a generator of the (unique) subgroup of order 3 of $\Z_p^*$. We may define two morphisms from $F$ to $\Z_p^*$, with kernel $\langle a^3,b \rangle$, namely

    $$\tau_1: (a^4,\rho^2) \mapsto \zeta_3, \quad \tau_2: (a^4,\rho^2) \mapsto \zeta_3^{-1}.$$

\noindent
We note that $(a^4,\rho^2)^{-1}=(a^2,\rho^4)=\sigma(a^4,\rho^2)\sigma$, hence $\tau_1=\tau_2 \Phi_{\sigma}$ and we obtain one brace.
\end{enumerate}

We state the obtained result in the following proposition.

\begin{proposition}
 Let $p\geq 7$ be a prime number. We count the left braces with additive group
     $\Z_p\times \Z_{6}\times \Z_2$ and multiplicative group $\Z_p\rtimes A_4$.
 \begin{enumerate}[1)]
     \item If $p\not \equiv 1 \pmod{3}$ there is just one such brace, which is a direct product.
      \item If $p\equiv 1 \pmod{3}$ there are 2 such braces. One is a direct product and the second one has kernel isomorphic to $V_4$.
 \end{enumerate}
 \end{proposition}

To be used in Section \ref{HG12p}, we compute $S_0(\tau)=\{ g \in \Aut F \, | \, \tau g=\tau \}$ for the two nontrivial morphisms from $F$ to $\Z_p^*$. We have $\Aut A_4 \simeq S_4$ and the isomorphism is obtained by sending a permutation in $S_4$ to the corresponding conjugation automorphism. We obtain $S_0(\tau_1)=S_0(\tau_2)=A_4$.

\subsection{$F=D_{2\cdot 6}$}

Let us write $F=\langle r,s \mid r^6=\Id, s^2=\Id, srs=r^5 \rangle$. We describe the morphisms $\tau:F \rightarrow \Z_p^*$. To be used in Section \ref{HG12p}, we compute $S_0(\tau)=\{ g \in \Aut F \, | \, \tau g=\tau \}$. We have $\Aut D_{2\cdot 6}=\langle \rho,\sigma \rangle \simeq D_{2\cdot 6}$, where $\rho$ and $\sigma$ are defined as follows.

$$\begin{array}{llll} \rho:& r & \mapsto & r \\ & s & \mapsto & rs \end{array}, \quad \begin{array}{llll} \sigma:&r & \mapsto & r^5 \\ & s& \mapsto & s \end{array}.$$

The only nontrivial morphisms from $F$ to $\Z_p^*$  are three morphisms with kernel of order 6, namely

$$\begin{array}{lllr} \tau_1:& r & \mapsto & 1 \\ & s & \mapsto & -1 \end{array}, \quad \begin{array}{lllr} \tau_2:&r & \mapsto & -1 \\ & s& \mapsto & -1 \end{array}, \quad \begin{array}{lllr} \tau_3:& r & \mapsto & -1 \\ & s & \mapsto & 1 \end{array},$$

\noindent
with kernels $\langle r\rangle, \langle r^2,rs \rangle$ and $\langle r^2,s\rangle$, respectively. We observe that $\Ker \tau_1$ is cyclic, while $\Ker \tau_2$ and $\Ker \tau_3$ are isomorphic to the dihedral group $D_{2\cdot 3}$. We have $S_0(\tau_1)=\Aut F, S_0(\tau_2)=S_0(\tau_3)=\langle \rho^2,\sigma \rangle$.

\subsubsection*{Case $E=C_{12}$}

There are two regular subgroups of $\mathrm{Hol}(E)$ isomorphic to $D_{2\cdot6}$ $$F_1=\langle\alpha_1=(2,1),\beta_1=(1,11)\rangle,\quad F_2=\langle\alpha_2=(1,7),\beta_2=(3,11)\rangle.$$ For $i\in\{1,2\}$, $\alpha_i$ has order $6$, $\beta_i$ has order $2$, and $\alpha_i\beta_i\alpha_i=\beta_i$, so $F_i\cong D_{2\cdot 6}$. It is checked easily that $F_i$ is regular. We have now $r=\alpha_i, s=\beta_i, i=1,2$.

We consider the morphisms from $F$ to $\Z_p^*$ with kernel of order 6.
 Since $\Ker(\tau_1)$ is cyclic while $\Ker \tau_2$ and $\Ker \tau_3$ are not,  $\tau_1$ is not conjugated to the other two morphisms. We denote $\tau_2^{(i)}, \tau_3^{(i)}:F_i \rightarrow \Z_p^*, i=1,2$.  Since $\Phi_7(\alpha_1)=\alpha_1$ and $\Phi_7(\beta_1)=\alpha_1^3 \beta_1$, we obtain $\tau_2^{(1)}=\tau_3^{(1)}\Phi_7$. For $\tau_2^{(2)}$ and $\tau_3^{(2)}$ to be conjugated, we would need $\Phi_{\nu}(\beta_2)=\alpha_2^k\beta_2$, with an odd $k$. Since the second component of $\beta_2$ is 11 and the second component of $\alpha_2^k \beta_2$ is 5, for an odd $k$, there is no such $\Phi_{\nu}$. Hence $\tau_2^{(2)}$ and $\tau_3^{(2)}$ are not conjugated and we obtain five braces, two of which have order 6 cyclic kernel.

\begin{proposition}
 Let $p\geq 7$ be a prime number. Then there are 7 left braces with additive group $\Z_p\times C_{12}$ and multiplicative group $\Z_p\rtimes D_{2\cdot 6}$. Among these, two of them are a direct product,  two other have cyclic kernel of order 6 and the other three have kernel isomorphic to $D_{2\cdot3}$.
\end{proposition}

\subsubsection*{Case $E=C_6\times C_2$}

If $E=C_6\times C_2$, we may take $F=\langle (a,\Id),(b,\rho^3) \rangle \subset \Hol(E)$, which is regular. Indeed, one may check that $(a,\Id)$ is of order $6$, $(b,\rho^3)$ is of order $2$, $(a,\Id)(b,\rho^3)(a,\Id)=(b,\rho^3)$ and $F$ has trivial stabilizer. We have now $r=(a,\Id), s=(b,\rho^3)$.

We consider the morphisms from $F$ to $\Z_p^*$ with kernel of order 6.
Again, since $\Ker(\tau_1)\cong C_6$ and $\Ker(\tau_i)\cong D_{2\cdot3}$, $i\in\{2,3\}$, $\tau_1$ is not conjugated to the other two morphisms. Since $\tau_2\Phi_{\rho^5}=\tau_3$, we obtain one brace with cyclic kernel and one brace with dihedral kernel.

\begin{proposition}
 Let $p\geq 7$ be a prime number. Then there are 3 left braces with additive group $\Z_p\times (C_{6}\times C_2)$ and multiplicative group $\Z_p\rtimes D_{2\cdot 6}$. Among these, one of them is a direct product, one has cyclic kernel of order 6 and the other one has kernel isomorphic to $D_{2\cdot3}$.
\end{proposition}

\subsection{$F=\Dic_{12}$}

The dicyclic group $\Dic_{12}$ is a group with 12 elements that can be presented as
$$\Dic_{12}=\langle x,y \, | \, x^3=1, y^4=1, yxy^{-1}=x^2\rangle.$$

We determine now the possible morphisms $\tau:F \rightarrow \Z_p^*$. To be used in Section \ref{HG12p}, we compute $S_0(\tau)=\{ g \in \Aut F \, | \, \tau g=\tau \}$. We have $\Aut \Dic_{12}=\langle \rho,\sigma \rangle \simeq D_{2\cdot 6}$, where $\rho$ and $\sigma$ are defined as follows.

$$\begin{array}{llll} \rho:& x & \mapsto & x \\ & y & \mapsto & xy^{-1} \end{array}, \quad \begin{array}{llll} \sigma:&x & \mapsto & x^{-1} \\ & y& \mapsto & y \end{array}.$$

\begin{enumerate}[1)]
\item There is a unique morphism $\tau$ from $F$ to $\Z_p^*$ with kernel of order 6, namely the one sending the generator $x$ to $1$ and $y$ to $-1$. We have $S_0(\tau)=\Aut F$.

\item  If $p \equiv 1 \pmod{4}$, let $\zeta_4$ be a generator of the subgroup of order 4 of $\Z_p^*$. We may define two morphisms from $F$ to $\Z_p^*$ with  kernel $\langle x \rangle$:
     $$\begin{array}{rlrl}
   \tau_1: &  x \mapsto 1  \qquad \qquad \qquad & \tau_2: & x \mapsto 1\\[2mm]
           & y \mapsto \zeta_4 &    & y \mapsto \zeta_4^{-1}.
    \end{array}$$

\noindent
We have $S_0(\tau_1)=S_0(\tau_2)=\langle \rho^2,\sigma \rangle$.

\end{enumerate}

\subsubsection*{Case $E=C_{12}$}

We know that in $\Hol(C_{12})$ there exists only a subgroup isomorphic to $F$ and it is regular.
We may take

$$F=\langle x=(4,1),y=(1,5) \rangle \subset \Hol(E),$$

\noindent
following the notation in Remark \ref{hol}. The element $x$ has order 3, the element $y$ has order 4 and they satisfy the relation $yxy^{-1}=x^2$. We may check that $F$ is a regular subgroup of $\Hol(C_{12})$.

We determine now the conjugation relations between the morphisms $\tau:F \rightarrow \Z_p^*$.

For the morphisms from $F$ to $\Z_p^*$ with  kernel $<x>$, we observe that $\tau_2=\tau_1 \Phi_{7}$, so we obtain, in this case, only one brace.

We state the obtained result in the following proposition.

\begin{proposition}
 Let $p\geq 7$ be a prime number. We count the left braces with additive group
     $\Z_p\times C_{12}$ and multiplicative group $\Z_p\rtimes  \Dic_{12}$.
 \begin{enumerate}[1)]
     \item If $p\not \equiv 1 \pmod{4}$ there are 2 such braces. One of them is a direct product and the other one has a kernel of order 6.
     \item If $p\equiv 1 \pmod{4}$ there are 3 such braces. One of them is a direct product, and the other two have kernels of order 6 and 3, respectively.
   \end{enumerate}
\end{proposition}

\subsubsection*{Case $E=C_6\times C_2$}

If $E=C_6\times C_2$, there is only a conjugacy class (of length 3)  of subgroups isomorphic to $\Dic_{12}$ that is regular.

We may take

$$F=\langle x=(a^2,Id),y=(b,\sigma) \rangle \subset \Hol(E),$$

\noindent
following the notation in Remark \ref{hol}.
 The element $x$ has order 6, the element $y$ has order 4 and they satisfy the relation $yxy^{-1}=x^2$.  We may check that $F$ is a regular subgroup of $\Hol(C_6\times C_2)$.

We determine now the conjugation relations between the  morphisms $\tau:F \rightarrow \Z_p^*$.

For the morphisms $\tau$ with a kernel of order 3, we observe that $\tau_2=\tau_1 \Phi_{\sigma}$, so we obtain, in this case, only one brace.

We state the obtained result in the following proposition.

\begin{proposition}
 Let $p\geq 7$ be a prime number. We count the left braces with additive group
     $\Z_p\times (C_{6}\times C_2)$ and multiplicative group $\Z_p\rtimes \Dic_{12}$.

 \begin{enumerate}[1)]
     \item If $p\not \equiv 1 \pmod{4}$ there are 2 such braces. One of them is a direct product and the other one has a kernel of order 6.
     \item If $p\equiv 1 \pmod{4}$ there are 3 such braces. One of them is a direct product, and the other two have kernels of order 6 and 3, respectively.

 \end{enumerate}
\end{proposition}

\subsection{Total numbers}

For a prime number $p\geq 7$ we compile in the following tables the total number of left braces of size $12p$.

The
additive group is $\Z_p\times E$ and the multiplicative group is a semidirect product $\Z_p\rtimes F$.
In the first column we have the possible $E$'s and in the first row the possible $F$'s.

\newpage
\begin{itemize}
    \item If $p\equiv 11 \pmod{12}$
\begin{center}
\begin{tabular}{r|c|c|c|c|c|c}
&\, \, $C_{12}$ \, \,   &   $C_6\times C_2$   & \, \,  $A_4$ \, \,   & \, \, $D_{2\cdot 6}$\, \,  & \, $\Dic_{12}$ \, \, & \\
\hline
   $C_{12}$   &2 &3 & 0 & 7 & 2&\\
  $C_6\times C_2$  &2 & 2&1 & 3 &2&\\

  \hline
  &4& 5 &1  & 10 & 4 & $\mathbf{24}$\\
\end{tabular}
\end{center}

    \item If $p\equiv 5 \pmod{12}$
\begin{center}
\begin{tabular}{r|c|c|c|c|c|c}
&\, \, $C_{12}$ \, \,   &   $C_6\times C_2$   & \, \,  $A_4$ \, \,   & \, \, $D_{2\cdot 6}$\, \,  & \, $\Dic_{12}$ \, \, & \\
\hline
   $C_{12}$   &3 &3 & 0 & 7 &3&\\
  $C_6\times C_2$  &3 &2 &1 & 3&3&\\

  \hline
  &6& 5 & 1 & 10 & 6 & $\mathbf{28}$\\
\end{tabular}
\end{center}

    \item If $p\equiv 7 \pmod{12}$
\begin{center}
\begin{tabular}{r|c|c|c|c|c|c}
&\, \, $C_{12}$ \, \,   &   $C_6\times C_2$   & \, \,  $A_4$ \, \,   & \, \, $D_{2\cdot 6}$\, \,  & \, $\Dic_{12}$ \, \, & \\
\hline
   $C_{12}$   & 4&6 & 0& 7 &2&\\
  $C_6\times C_2$  & 4& 4&2 &3 &2&\\

  \hline
  &8& 10 &2  & 10 & 4 & $\mathbf{34}$\\
\end{tabular}
\end{center}

    \item If $p\equiv 1 \pmod{12}$
\begin{center}
\begin{tabular}{r|c|c|c|c|c|c}
&\, \, $C_{12}$ \, \,   &   $C_6\times C_2$   & \, \,  $A_4$ \, \,   & \, \, $D_{2\cdot 6}$\, \,  & \, $\Dic_{12}$ \, \, & \\
\hline
   $C_{12}$   &6 &6 & 0& 7 &3&\\
  $C_6\times C_2$  &6 &4 & 2&3 &3&\\

  \hline
  &12& 10 &2  & 10 & 6 & $\mathbf{40}$\\
\end{tabular}
\end{center}

\end{itemize}

With the results summarized in the above tables, the validity of conjecture (\ref{conj}) is then established.

\section{Hopf Galois structures on a Galois field extension of degree $12p$}\label{HG12p}

Let $E, F$ be groups of order 12 with $E$ abelian. By computation with Magma, we obtain that the number of regular subgroups of  $\Hol(E)$ isomorphic to $F$ is as shown in the following table.

\vspace{0.5cm}
\begin{center}
\begin{tabular}{|c||c|c|c|c|c|}
\hline
$E \backslash F$ & \, \, $C_{12}$ \, \, & $C_6\times C_2$ & \, \, $A_4$ \, \, \, & \, \, $D_{2\cdot 6}$ \, \, & \,  $\Dic_{12}$ \, \, \\
\hline
\hline
$C_{12}$ &1&1&0&3& 1\\
\hline
$C_6\times C_2$ &3&1&2&3& 3\\
\hline
\end{tabular}
\end{center}

More precisely, for the groups $F_1, F_2$ defined in the case $F=D_{2\cdot 6}, E=C_{12}$, we obtain that $F_1$ is normal in $\Hol(E)$ while the length of the conjugation class of $F_2$ in $\Hol(E)$ is 2 and $F_2'=\langle (7,7),(9,11)\rangle$ is the second subgroup in this class. Using Byott's formula, Corollary \ref{autf} and the determination of $S_0$ given in Section \ref{braces}, we obtain  number of Hopf Galois structures of abelian type on a Galois field extension of degree $12p$.

The number of Hopf Galois structures of abelian type on a Galois extension with Galois group $G=\Z_p \rtimes_{\tau} F$ is as given in the following tables. The first column gives the group $F$ and the first row the kernel of the morphism $\tau:F \rightarrow \Z_p^*$ defining the semidirect product. In each case, we assume that the value of $p$ is such that a morphism $\tau:F \to \Z_p$ exists with the given kernel.

\begin{center}
{\bf Hopf Galois structures of type $C_{12p}$}

\vspace{0.3cm}
\begin{tabular}{|c||c|c|c|c|c|c|c|c|}
\hline
$F \backslash \Ker \tau$ & \,  $F$ \,  & \, $C_6$\,  & $D_{2\cdot 3}$  &\,  $C_4$ \, & \, $C_2^2$\,  & \, $C_3$\,  & \, $C_2$ \, &\, $\{1 \}$ \, \\
\hline
\hline
$C_{12}$ &1&$p$&-&$p$&-&$p$&$p$&$p$ \\
\hline
$C_6\times C_2$ &3&$3p$&-&-& $3p$&-&$3p$&-\\
\hline
$A_4$ &0&-&-&-&0&-&-&- \\
\hline
$D_{2\cdot 6}$ &9&$9p$&$9p$&-&-&-&-&-\\
\hline
$\Dic_{12}$ &3&$3p$&-&-&-&$3p$&-&-\\
\hline
\end{tabular}
\end{center}

\vspace{0.5cm}
\begin{center}
{\bf Hopf Galois structures of type $C_{6p}\times C_2$}

\vspace{0.3cm}
\begin{tabular}{|c||c|c|c|c|c|c|c|c|}
\hline
$F \backslash \Ker \tau$ & \,  $F$ \,  & \, $C_6$\,  & $D_{2\cdot 3}$  &\,  $C_4$ \, & \, $C_2^2$\,  & \, $C_3$\,  & \, $C_2$ \, &\, $\{1 \}$ \, \\
\hline
\hline
$C_{12}$ &1&$p$&-&$p$&-&$p$&$p$&$p$ \\
\hline
$C_6\times C_2$ &1&$p$&-&-& $p$&-&$p$&-\\
\hline
$A_4$ &4&-&-&-&$4p$&-&-&- \\
\hline
$D_{2\cdot 6}$ &3&$3p$&$3p$&-&-&-&-&-\\
\hline
$\Dic_{12}$ &3&$3p$&-&-&-&$3p$&-&-\\
\hline
\end{tabular}
\end{center}

\end{document}